\documentclass[12pt,oneside,english,reqno]{amsart}

\usepackage[T1]{fontenc}
\usepackage[latin9]{inputenc}
\usepackage[letterpaper]{geometry}
\usepackage{mathrsfs}
\usepackage{amstext}
\usepackage{amsthm}
\usepackage{amssymb}
\usepackage{setspace}
\usepackage{esint}
\makeatletter
\numberwithin{equation}{section}
\numberwithin{figure}{section}
\theoremstyle{plain}
\newtheorem{thm}{\protect\theoremname}[section]
  \theoremstyle{remark}
  \newtheorem{rem}[thm]{\protect\remarkname}
  \theoremstyle{plain}
  \newtheorem{assumption}[thm]{\protect\assumptionname}
  \theoremstyle{plain}
  \newtheorem{prop}[thm]{\protect\propositionname}
  \theoremstyle{plain}
  \newtheorem{cor}[thm]{\protect\corollaryname}
  \theoremstyle{plain}
  \newtheorem{lem}[thm]{\protect\lemmaname}

\usepackage{dsfont}
\usepackage{fontenc}

\makeatother

\usepackage{babel}
  \providecommand{\assumptionname}{Assumption}
  \providecommand{\corollaryname}{Corollary}
  \providecommand{\lemmaname}{Lemma}
  \providecommand{\propositionname}{Proposition}
  \providecommand{\remarkname}{Remark}
\providecommand{\theoremname}{Theorem}

\begin{document}

\title{Scaling Limit of Two-component Interacting Brownian Motions}

\author{Insuk Seo}

\address{Department of Mathematics, University of California Berkeley, 970
Evans Hall, Berkeley, CA 94720-3840 USA}

\email{insuk@berkeley.edu }


\begin{abstract}
This paper presents our study of the asymptotic behavior of a two-component
system of Brownian motions undergoing certain form of singular interactions.
In particular, the system is a combination of two different types
of particles and the mechanical properties and the interaction parameters
depend on the corresponding type of particles. We prove that the hydrodynamic
limit of the empirical densities of two types is the solution of a
partial differential equation known as the Maxwell--Stefan equation. 
\end{abstract}

\maketitle

\section{Introduction\label{s1}}

In this article, we are primarily concerned with the large-scale analysis
of the locally interacting Brownian motions (LIBM), which consists
of two different types of particles. Originally, the one-component
LIBM consisting of identical particles was introduced by \cite{G1,G2},
in which the limit theorem for the tagged particle in non-equilibrium
dynamics and the propagation of chaos was established. Recently, in
\cite{S} the author of the current paper developed a large deviation
principle for the bulk average of particle trajectories of the one-component
LIBM by analyzing the so-called two-color system. In the two-color
system, all particles have the same mechanical property but each particle
is painted by either black or white. The analysis of this system is
known to be a difficult task because of the so-called non-gradient
property, and accordingly, the limit theorem and the large deviation
principle for the two-color system are known only for three interacting
particle systems: the symmetric simple exclusion process (SSEP) \cite{Q},
the zero-range process (ZRP) \cite{GJL} and the LIBM \cite{S}. 

Furthermore, we can also consider the two-component system, instead
of the two-color system, in which particles of different colors have
different mechanical properties. Few results are announced for the
two-component SSEP with additional interaction mechanisms. For instance,
a two-component SSEP where two types of particles affect each other
through their jump rates is studied in \cite{FGQ}. A two-component
SSEP under the presence of simultaneous births, deaths, and switching
of different types of particles is investigated in \cite{Sas}. Recently,
a two-component weakly asymmetric exclusion process where the type
of each particle is randomly updated according to the types of its
neighboring particles is analyzed in \cite{E}. However, the hydrodynamic
limit of the two-component SSEP without additional interaction mechanism
other than the exclusion dynamics is an open problem. This kind of
result is obtained for the ZRP. In \cite{GJL2}, the hydrodynamic
limit for the two-component ZRP without additional interaction mechanism
other than zero-range dynamics is obtained. In this paper, we develop
the second result of this kind for the two-component LIBM. We also
provide extensive discussions for general features and technical difficulties
for two-component interacting particle systems.

We also remark here that another two-component system under the recent
attention is the chain of harmonic oscillators \cite{KOS, Sim}. For
these models, two conserved quantities are the energy and the volume
of the system.

\subsection{Outline}

In Section \ref{s2}, we introduce a precise definition of the particle
system and state our main results. We also discuss the main feature
of the model, which enables us to obtain the hydrodynamic limit although
the system is non-gradient. In Section \ref{s3}, we establish the
hydrodynamic limit of the system, which amounts to the main result.
The hydrodynamic limit equation is a system of parabolic equations
\eqref{em1} with the the explicit cross-diffusion matrix $D$ (cf.
(\ref{em2})), where the diffusion matrix $D$ is not symmetric, and
$\frac{1}{2}(D+D^{\dagger})$ is not even positive-definite. For this
equation, the uniqueness of the weak solution is known to be a delicate
problem, and the general theory \cite{A} only provides the local
uniqueness. Consequently, the hydrodynamic limit result for the general
initial condition is local in time. The global result is achieved
only for the initial condition that is sufficiently close to the equilibrium.
This finding is a common feature of multi-component interacting particle
systems, including the ZRP, due to the result of \cite{GJL2}. We
discuss these uniqueness issues in Section \ref{s4}. We finally remark
here that the hydrodynamic limit equation \eqref{em1} for the two-component
LIBM is the so-called \textit{Maxwell--Stefan equation,} which describes
multi-component gaseous mixtures and is explained in Section \ref{s43}.

\section{Model and Main Result\label{s2}}

\subsection{\label{s21}Type of particles}

Consider a system of $N$ interacting particles $x_{1}^{N}(\cdot),\,\cdots,\,x_{N}^{N}(\cdot)$
on $\mathbb{T}$ where each particle belongs to one of the two given
types, namely, type $1$ and type $2$. We denote by $\mathcal{T}_{c}^{N}\subset\{1,\,2,\,\cdots,\,N\},$
the index sets of type $c\in\{1,\,2\}$. Specifically, the set of
particles of type $c\in\{1,\,2\}$ is $\{x_{i}^{N}(t):i\in\mathcal{T}_{c}^{N}\}$.
The diffusivity of the particle of type $c$ is $\sigma_{c}^{2}>0$
and our primary interest is the non-homogeneous case: $\sigma_{1}^{2}\neq\sigma_{2}^{2}$.
If $\sigma_{1}^{2}=\sigma_{2}^{2}$, the system becomes the one-component,
two-color system considered in \cite{S}. 

In addition, the interaction mechanism also depends on the type of
particles. In the local interaction model, two particles always reflect
each other regardless of their types when they collide, but they sometimes
change their labels. To explain this switching mechanism more precisely,
we first measure the amount of collision between two particles $x_{i}^{N}(\cdot)$
and $x_{j}^{N}(\cdot)$ up to time $t$ by local times $A_{ij}^{N}(t)$
and $A_{ji}^{N}(t)$, depending on their relative positions infinitesimally
before the collision:
\begin{align}
A_{ij}^{N}(t) & =\lim_{\epsilon\rightarrow0}\int_{0}^{t}\frac{\mathds{1}_{[0,\epsilon]}(x_{i}^{N}(s)-x_{j}^{N}(s))}{2\epsilon}ds\;,\label{loc1}\\
A_{ji}^{N}(t) & =\lim_{\epsilon\rightarrow0}\int_{0}^{t}\frac{\mathds{1}_{[-\epsilon,0]}(x_{i}^{N}(s)-x_{j}^{N}(s))}{2\epsilon}ds\;,\label{loc}
\end{align}
where the limit exists almost surely, e.g., see \cite[Chapter 2]{IM}.
In the one-component model \cite{G1,S}, the label switching between
two particles $x_{i}^{N}(\cdot)$ and $x_{j}^{N}(\cdot)$ occurs as
a Poisson process with constant intensity $\lambda N$ along this
local time clock. The main difference for the two-component system
is the dependence of the interaction parameter $\lambda$ on the types
of particles involved. Let $\lambda_{c_{1},c_{2}}\ge0$, $c_{1},\,c_{2}\in\{1,\,2\}$,
be four (possibly different) constants. Then, the label switching
between $x_{i}^{N}(\cdot)$ of type $c_{1}$ and $x_{j}^{N}(\cdot)$
of type $c_{2}$, occurs as the Poisson process with intensity $\lambda_{c_{1},c_{2}}N$
along the local time $A_{ij}^{N}(t)$ and with intensity $\lambda_{c_{2},c_{1}}N$
along the local time $A_{ji}^{N}(t)$. The rigorous definition of
the model described above will be given in the next subsection. 
\begin{rem}
Our primary interest in the current paper is the symmetric case, i.e.,
$\lambda_{1,2}=\lambda_{2,1}$, which shall be explicitly remarked
in Section \ref{s22}. On the other hand, our construction of the
process presented in Section \ref{s212} is valid without this constraint. 
\end{rem}

\subsection{Rigorous formulation of two-component LIBM\label{s212}}

The rigorous construction of the model described above can be carried
out in a similar manner to the one-component system \cite{G1,S}.
Let $G_{N}$ be the $N$-manifold
\begin{equation}
G_{N}=\{\mathbf{x}=(x_{1},\,x_{2},\,\cdots,\,x_{N})\in\mathbb{T}^{N}:x_{i}\neq x_{j}\,\,\text{for all }i\neq j\}\;,\label{e1}
\end{equation}
then the LIBM is constructed as a diffusion process on $G_{N}$ with
reflecting boundary condition. We first characterize the boundary
$\partial G_{N}$ of $G_{N}$. Any point $\mathbf{x}\in\partial G_{N}$
must satisfy $x_{i}=x_{j}$ for some $i\neq j$. However the face
$\{\mathbf{x}:x_{i}=x_{j}\}$, $i\neq j$, has two sides $F_{ij}$
and $F_{ji}$ corresponding to the boundary of two disconnected sets
$U_{ij}=\{\mathbf{x}\in G_{N}:x_{i}<x_{j}\}$ and $U_{ji}=\{\mathbf{x}\in G_{N}:x_{j}<x_{i}\}$,
respectively. Specifically, $F_{ij}$, $i\neq j$, can be regarded
as the equivalence class on sequences $(\mathbf{x}_{n})_{n=1}^{\infty}\mbox{ in }U_{ij}$
which converges to some point $\mathbf{x}$ satisfying $x_{i}=x_{j}$.
The equivalent class $\sim$ is defined by $(\mathbf{x}_{n})_{n=1}^{\infty}\sim(\widetilde{\mathbf{x}}_{n})_{n=1}^{\infty}$
if $\lim\mathbf{x}_{n}=\lim\widetilde{\mathbf{x}}_{n}$. Then, the
boundary $\partial G_{N}$ can be written as $\partial G_{N}=\bigcup_{i\neq j}F_{ij}$.
Denote by $\overline{\mathcal{C}}(G_{N})$ the set of smooth functions
on $G_{N}$ that are smooth up to the boundary $\partial G_{N}$.
For $1\le i\le N$, denote by $c(i)$ the type of particle $x_{i}^{N}(\cdot)$.
For $f\in\overline{\mathcal{C}}(G_{N})$ and $\mathbf{x}\in F_{ij}$,
define
\begin{align}
f_{ij}(\mathbf{x}) & =\lim_{\mathbf{x}_{n}\rightarrow\mathbf{x},\,\mathbf{x}_{n}\in U_{ij},\,\forall n}f(\mathbf{x}_{n})\;,\label{e11}\\
(D_{ij}f)(\mathbf{x}) & =\lim_{\mathbf{x}_{n}\rightarrow\mathbf{x},\,\mathbf{x}_{n}\in U_{ij},\,\forall n}(\sigma_{c(i)}^{2}\partial_{i}-\sigma_{c(j)}^{2}\partial_{j})f(\mathbf{x}_{n})\;,\label{e12}
\end{align}
so that $f_{ij}$ and $D_{ij}f$ are functions on $F_{ij}$. 

The LIBM is a diffusion process $\mathbf{x}^{N}(\cdot)=(x_{1}^{N}(\cdot),\,\cdots,\,x_{N}^{N}(\cdot))$
on $G_{N}$ with the pregenerator 
\begin{equation}
\mathscr{L}_{N}f=\frac{\sigma_{1}^{2}}{2}\sum_{i\in\mathcal{T}_{1}^{N}}\partial_{i}^{2}f+\frac{\sigma_{2}^{2}}{2}\sum_{j\in\mathcal{T}_{2}^{N}}\partial_{j}^{2}f=\frac{1}{2}\sum_{i=1}^{N}\sigma_{c(i)}^{2}\partial_{i}^{2}f\;,\label{gen}
\end{equation}
where the domain $\mathcal{D}(\mathscr{L}_{N})$ of $\mathscr{L}_{N}$
consists of functions $f\in\overline{\mathcal{C}}(G_{N})$ satisfying
the boundary condition $(\mathscr{B}_{ij}f)(\mathbf{x})=0$ on $F_{ij}$
for all $i\neq j$, where $\mathscr{B}_{ij}f$ is a function on $F_{ij}$
defined by 
\begin{equation}
\mathscr{\mathscr{B}}_{ij}f=D_{ij}f-\lambda_{c(i),c(j)}N(f_{ij}-f_{ji})\;.\label{e2}
\end{equation}
In \eqref{e2}, the first term corresponds to the reflection between
two particles $x_{i}^{N}(\cdot)$ and $x_{j}^{N}(\cdot)$, while the
second term explains the switching of labels between the two particles.
The Lebesgue measure $d\mathbf{x}$ on $G_{N}$ is the invariant measure
for $\mathscr{L}_{N}$, and the process $\mathbf{x}^{N}(\cdot)$ is
reversible with respect to $d\mathbf{x}$. 

Alternative construction of the particle system can be achieved by
the martingale formulation of the diffusion processes reflecting at
the boundary \cite{IW}, as noted in \cite{G1}. More precisely, for
the fixed final time $T>0$, the diffusion process $\mathbf{x}^{N}(\cdot)$
on $G_{N}$ that we constructed above can be regarded as a probability
measure $\mathbb{P}_{N}$ on $C([0,T],G_{N})$. Then there exist local
times $A_{ij}^{N}(t)$, $1\le i\neq j\le N$, such that, for all $f\in\overline{\mathcal{C}}(G_{N})$,
\begin{equation}
f(\mathbf{x}^{N}(t))-f(\mathbf{x}^{N}(0))-\int_{0}^{t}\sum_{i=1}^{N}\frac{\sigma_{c(i)}^{2}}{2}(\partial_{i}^{2}f)(\mathbf{x}^{N}(s))ds-\sum_{i\neq j}\int_{0}^{t}(\mathscr{\mathscr{B}}_{ij}f)(\mathbf{x}^{N}(s))dA_{ij}^{N}(s)\label{e3}
\end{equation}
is a $(\mathbb{P}_{N},\{\mathscr{F}_{t}:0\le t\le T\})$ martingale
where $\mathscr{F}_{t}=\sigma(\mathbf{x}^{N}(s):0\le s\le t)$. The
martingale \eqref{e3} admits the following alternative expression:
\begin{equation}
\sum_{i=1}^{N}\sigma_{c(i)}\int_{0}^{t}(\partial_{i}f)(\mathbf{x}^{N}(s))d\beta_{i}(s)+\sum_{i\neq j}\int_{0}^{t}(f_{ij}-f_{ji})(\mathbf{x}^{N}(s))\left[dM_{ij}^{N}(s)-dM_{ji}^{N}(s)\right]\;,\label{e4}
\end{equation}
where $\{\beta_{i}(\cdot):1\le i\le N\}$ is a family of independent
Brownian motions and $\{M_{ij}^{N}(\cdot):1\le i\neq j\le N\}$ is
a family of pairwise orthogonal compensated Poisson jump processes
where the rate of process $M_{ij}^{N}(t)$ is $\lambda_{c(i),c(j)}NA_{ij}^{N}(t)$
for all $i\neq j$. We denote the expectation with respect to $\mathbb{P}_{N}$
as $\mathbb{E}_{N}$.
\begin{rem}
\label{r21}The particle system consisting of $m$ different types
$1,\,2,\,\cdots,\,m$ can be defined similarly. 
\end{rem}

\subsection{\label{s22}Main result }

The empirical density at time $t$ of particles of type $c\in\{1,2\}$
is defined by 
\[
\mu_{c}^{N}(t)=\frac{1}{N}\sum_{i\in\mathcal{T}_{c}^{N}}\delta_{x_{i}^{N}(t)}\in\mathscr{M}(\mathbb{T})\;,
\]
and then the empirical density at time $t$ of the entire system can
be written as 
\[
\widetilde{\mu}^{N}(t)=\left(\mu_{1}^{N}(t),\,\mu_{2}^{N}(t)\right)^{\dagger}\in\mathscr{M}(\mathbb{T})^{2}\;.
\]
Fix the final time $T$ and let $\mathbb{Q}_{N}$ be the measure on
$C([0,T],\mathscr{M}(\mathbb{T})^{2})$ induced by the process $\widetilde{\mu}^{N}(\cdot)$.
Then, our goal is to characterize the limit point of $\{\mathbb{Q}_{N}\}_{N=1}^{\infty}$
as a Dirac measure on the unique solution of a certain partial differential
equation (PDE). To specify the initial condition of the PDE, we assume
that the initial empirical density $\widetilde{\mu}^{N}(0)$ satisfies
the law of large number in the sense that 
\begin{equation}
\widetilde{\mu}^{N}(0)\rightharpoonup\widetilde{\rho}^{\,0}(x)dx=\left(\rho_{1}^{0}(x)dx,\,\rho_{2}^{0}(x)dx\right)^{\dagger}\label{e7}
\end{equation}
weakly in $\ensuremath{\ensuremath{\mathscr{M}}(\mathbb{T})^{2}}$
for some non-negative initial density functions $\rho_{1}^{0}$ and
$\rho_{2}^{0}$. By pairing this weak convergence with the constant
function $1$, we derive $\left|\mathcal{T}_{c}^{N}\right|=N(\overline{\rho}_{c}+o_{N}(1))$
where $\overline{\rho}_{c}=\int_{\mathbb{T}}\rho_{c}^{0}(x)dx$ is
the average density of type $c$. 

We explain several technical assumptions to obtain the result in a
concrete form. The standard Sobolev space on $\mathbb{T}$ is denoted
by $W_{k,p}(\mathbb{T})$, and let $H_{k}(\mathbb{T})=W_{k,2}(\mathbb{T})$.
The following assumptions are required in the investigation of the
uniqueness result in Section \ref{s4}. The results in Section \ref{s3}
are valid without these assumptions.
\begin{assumption}
\noindent \label{a1} The function $\widetilde{\rho}^{\,0}(\cdot)$
belongs to $W^{1,p}(\mathbb{T})^{2}$ for some $p>2$. 
\end{assumption}

\begin{assumption}
\label{a2}The function $\widetilde{\rho}^{\,0}(\cdot)$ belongs to
$H_{2}(\mathbb{T})^{2}$, and satisfies
\[
\left\Vert \rho_{1}^{0}(\cdot)-\overline{\rho}_{1}\right\Vert _{H_{2}(\mathbb{T})}+\left\Vert \rho_{2}^{0}(\cdot)-\overline{\rho}_{2}\right\Vert _{H_{2}(\mathbb{T})}<\epsilon
\]
for some sufficiently small constant $\epsilon=\epsilon(\lambda,\sigma_{1},\sigma_{2})>0$.
This constant is explained at the end of Section \ref{s43}.
\end{assumption}
In addition, we also assume that the initial particle configuration
satisfies the following entropy condition, \cite{GPV,V2}. 
\begin{assumption}
\noindent \label{a3}The distribution of initial configuration $\mathbf{x}^{N}(0)=(x_{1}^{N}(0),\cdots,x_{N}^{N}(0))$
is absolutely continuous with respect to the Lebesgue measure $d\mathbf{x}$,
and the probability density function $f_{N}^{0}(\cdot)$ of $\mathbf{x}^{N}(0)$
satisfies 
\begin{equation}
\int_{G_{N}}f_{N}^{0}(\mathbf{x})\log f_{N}^{0}(\mathbf{x})d\mathbf{x}\le CN\;,\;\;\forall N\in\mathbb{N}\;,\label{e8}
\end{equation}
for some constant $C>0$ independent of $N$.
\end{assumption}
This assumption enables us to control the entropy production in Section
\ref{s311} and is required in the proof of the tightness of $\{\mathbb{Q}_{N}\}_{N=1}^{\infty}$
in Section \ref{s3}. Remark that the $i.i.d.$ configuration satisfies
Assumption \ref{a3}.

In this paper, we are interested in the symmetric case, i.e., $\lambda_{1,2}=\lambda_{2,1}$
only. Thus, for the sake of convenience, we shall normalize 
\begin{equation}
\lambda_{1,2}=\lambda_{2,1}=\lambda\sigma_{1}^{2}\sigma_{2}^{2}\;.\label{epa1}
\end{equation}
In particular, we are concerned with the asymptotic behavior of the
empirical density $\widetilde{\mu}^{N}(\cdot)$, which is not affected
by the switching of the labels between particles of the same type.
Hence, the interactions among the particles of same type certainly
do not affect $\widetilde{\mu}^{N}(\cdot)$ and in turn, neither the
assumptions nor the conclusion are influenced even if we assume that
the self-interaction parameters $\lambda_{1,1}$ and $\lambda_{2,2}$
satisfy
\begin{equation}
\lambda_{1,1}=\lambda\sigma_{1}^{4},\,\,\lambda_{2,2}=\lambda\sigma_{2}^{4}\;.\label{e31}
\end{equation}
We emphasize that all the results obtained in this article are remaining
in force for systems with any interacting parameters, $\lambda_{1,1},\,\lambda_{2,2}>0$.
Finally, we can write the interaction parameter between $x_{i}^{N}(\cdot)$
and $x_{j}^{N}(\cdot)$ by $\lambda\sigma_{c(i)}^{2}\sigma_{c(j)}^{2}$,
under the notations \eqref{epa1} and \eqref{e31}.

The following theorem is the main result. 
\begin{thm}
\label{t1} Under Assumption \ref{a3}, the family of probability
measures $\{\mathbb{Q}_{N}\}_{N=1}^{\infty}$ is tight and any of
its limit point is concentrated on the trajectory of the form $(\rho_{1}(\cdot,x)dx,\,\rho_{2}(\cdot,x)dx)^{\dagger}$,
where $(\rho_{1},\rho_{2})^{\dagger}$ is a weak solution of the equation
\begin{equation}
\frac{\partial}{\partial t}\begin{pmatrix}\rho_{1}\\
\rho_{2}
\end{pmatrix}=\frac{1}{2}\nabla\cdot\bigg[D(\rho_{1},\rho_{2})\nabla\begin{pmatrix}\rho_{1}\\
\rho_{2}
\end{pmatrix}\bigg]\;;\;\;t\in[0,T]\label{em1}
\end{equation}
with initial condition $\widetilde{\rho}^{\,0}(\cdot)=(\rho_{1}^{0}(\cdot),\rho_{2}^{0}(\cdot))^{\dagger}$,
and the cross-diffusion matrix $D(\cdot,\cdot)$ is given by 
\begin{equation}
D(\rho_{1},\rho_{2})=\frac{1}{\lambda+\frac{\rho_{1}}{\sigma_{1}^{2}}+\frac{\rho_{2}}{\sigma_{2}^{2}}}\begin{pmatrix}\rho_{1}+\lambda\sigma_{1}^{2} & \rho_{1}\\
\rho_{2} & \rho_{2}+\lambda\sigma_{2}^{2}
\end{pmatrix}\;.\label{em2}
\end{equation}
Furthermore, 
\begin{enumerate}
\item (Local result) Under Assumption \ref{a1}, there exists $t^{+}(\widetilde{\rho}^{\,0})>0$
such that equation (\ref{em1}) has a unique weak solution in $[0,t^{+}(\widetilde{\rho}^{\,0}))$.
Therefore, $\{\mathbb{Q}_{N}\}_{N=1}^{\infty}$ converges weakly to
the Dirac mass concentrated on this unique solution if $T\le t^{+}(\widetilde{\rho}^{\,0})$. 
\item (Global result) Under Assumption \ref{a2}, the weak solution of equation
(\ref{em1}) is unique for any $T>0$. Hence, $\{\mathbb{Q}_{N}\}_{N=1}^{\infty}$
converges weakly to the Dirac mass concentrated on this unique solution. 
\end{enumerate}
\end{thm}
We now briefly explain the main feature of the model which enables
us to compute the hydrodynamic limit with the explicit diffusion coefficient
\eqref{em2}. We first review the model with $\sigma_{1}=\sigma_{2}=1$,
which is the one-component, two-color system considered in \cite{S}.
In general, the two-color interacting particle system is non-gradient
in the sense of \cite{V2}, and this property makes the analysis of
hydrodynamic limit a complicated project. The robust non-gradient
method has been developed by Varadhan \cite{V2}, and its application
to the two-color SSEP was achieved by Quastel \cite{Q}. For the detailed
discussion on the non-gradient method, we refer to \cite[Chapter 7]{KLO}.

The one-component LIBM is verified to be a non-gradient system. Define
$\left\langle \mu,f\right\rangle =\int_{\mathbb{T}}fd\mu$. Then,
the so-called density field for type $c\in\{1,\,2\}$ corresponding
to the function $g\in C^{1}(\mathbb{T})$ is defined by $\left\langle \mu_{c}^{N}(t),g\right\rangle $.
Then, by \eqref{e3}, we obtain that 
\[
\left\langle \mu_{c}^{N}(t),g\right\rangle -\left\langle \mu_{c}^{N}(0),g\right\rangle =\frac{1}{N}\sum_{i\in\mathcal{T}_{1}^{N}}\int_{0}^{t}g'(x_{i}^{N}(s))\cdot N\sum_{k:k\neq i}\left[dA_{ik}^{N}(s)-dA_{ki}^{N}(s)\right]+M_{t}
\]
where $M_{t}$ is a martingale. Hence, in view of \eqref{loc1} and \eqref{loc},
the current around the particle $x_{i}^{N}(s)$ can be formally defined
by 
\begin{equation}
g'(x_{i}^{N}(s))\cdot N\sum_{k:k\neq i}\left[\delta_{+}(x_{k}^{N}(s)-x_{i}^{N}(s))-\delta_{-}(x_{i}^{N}(s)-x_{k}^{N}(s))\right]\label{cur}
\end{equation}
where $\delta_{+}$ and $\delta_{-}$ are right and left Dirac functions
at $0$, respectively. The LIMB can be easily observed as a non-gradient
system in this expression. The application of the non-gradient method
to the singular object, such as \eqref{cur}, became a technically
difficult issue. Furthermore, at the time when this paper was written,
the non-gradient method could not be applied to the non-gradient,
diffusion-type interacting particle system. For instance, the analysis
of the two-color system of interacting diffusions considered in \cite{V1}
is an open problem. However, the LIBM owns a particular feature, which
enables us to derive the hydrodynamic limit without appealing the
non-gradient method. This feature is briefly explained below. 

In \cite{KV}, Kipnis and Varadhan reduced the investigation of the
tagged particle in a reversible interacting particle system to a central
limit theorem for certain reversible Markov chains. We refer to \cite[Chapter 6]{KLO}
for the detailed exposition of this topic. This central limit theorem
is obvious if there exists a solution of a certain Poisson equation
of the form $-\mathcal{L}f=V$, e.g., see \cite[Chapter 1]{KLO} for
details. The general methodology of \cite{KV} deals with the situation
for which this Poisson equation does not have a solution. The main
feature of the one-component LIBM is the explicit, simple solution
in the corresponding Poisson equation. Hence, the argument of \cite{KV}
is not required. Given this feature, we can simplify the computation
considerably and compute the diffusion coefficient in an explicit
form. Furthermore, this feature is also useful in the examination
of the non-equilibrium tagged particle in \cite{G1}. More precisely,
in \cite[Definition 8]{G1} a martingale $z_{1}^{N}(t)$ is introduced
and is constructed by using the explicit solution of the Poisson equation.
The martingale $z_{1}^{N}(t)$ allows \cite{G1} and \cite{S} to detour
the non-gradient method. Although our model is two-component, we introduce
another form of martingale \eqref{e310} in the computation of hydrodynamic
limits. This martingale enables us to compute the explicit diffusion
coefficient \eqref{em2}.

\section{Hydrodynamic Limit\label{s3}}

In this section, we prove Theorem \ref{t1}. The uniqueness issue
pertaining to equation \eqref{em1} is discussed in the next section.

\subsection{Green's formula for $G_{N}$}

We briefly review Green's formula for $G_{N}$. Although the current
paper is self-contained, we refer to \cite[Section 2.2]{S} for detailed
explanation. 

Fix a function $u\in\overline{\mathcal{C}}(G_{N})$ and a smooth vector
field $\mathbf{V}$ on $G_{N}$. Recall that the boundary $\partial G_{N}$
can be expressed as $\bigcup_{i\neq j}F_{ij}$. Note that the normal
vector to $F_{ij}$ is $(\mathbf{e}_{i}-\mathbf{e}_{j})/\sqrt{2}$
where $\mathbf{e}_{i}$ represents the $i$th standard normal vector.
The Lebesgue measure on the surface $F_{ij}=\{\mathbf{\mathbf{x}}:x_{i}=x_{j}\}$,
which is denoted by $dS_{ij}(\mathbf{x})$, is normalized to have
a total measure of $1$. Then, Green's formula for $G_{N}$ implies
that
\begin{equation}
\int_{G_{N}}\nabla u(\mathbf{x})\cdot\mathbf{V}(\mathbf{x})d\mathbf{x}=-\int_{G_{N}}u(\mathbf{x})(\nabla\cdot\mathbf{V})(\mathbf{x})d\mathbf{x}+\sum_{i\neq j}\int_{F_{ij}}u_{ij}(\mathbf{x})\left[\mathbf{V}(\mathbf{x})\cdot(\mathbf{e}_{i}-\mathbf{e}_{j})\right]dS_{ij}(\mathbf{x})\;.\label{gr}
\end{equation}
Note that factor $1/\sqrt{2}$ disappeared because of the renormalization
of the measure $dS_{ij}$ on $F_{ij}$. This formula will be used
frequently hereafter.

\subsection{Entropy production\label{s311}}

Since the invariant measure for $\mathscr{L}_{N}$ is $d\mathbf{x},$
the corresponding Dirichlet form is defined by $\mathscr{D}_{N}(f)=\int_{G_{N}}f(\mathbf{x})(-\mathscr{L}_{N}f)(\mathbf{x})d\mathbf{x}$
for $f\in\overline{\mathcal{C}}(G_{N})$. For non-negative function
$f\in\overline{\mathcal{C}}(G_{N})$ define $\mathcal{D}_{N}(f)=\mathscr{D}_{N}(\sqrt{f})$.
Then, it is easy to verify that
\begin{equation}
\mathcal{D}_{N}(f)=\sum_{i=1}^{N}\int_{G_{N}}\frac{\sigma_{c(i)}^{2}}{8}\frac{(\partial_{i}f)^{2}}{f}d\mathbf{x}+\sum_{i\neq j}\frac{\lambda N\sigma_{c(i)}^{2}\sigma_{c(j)}^{2}}{2}\int_{F_{ij}}\left(\sqrt{f_{ij}}-\sqrt{f_{ji}}\right)^{2}dS_{ij}(\mathbf{x})\;.\label{e33}
\end{equation}
Denote by $f_{N}(t,\:\cdot)$ the marginal density of $\mathbf{x}^{N}(\cdot)$
at time $t$ with respect to $d\mathbf{x}$. Then, the entropy at
time $t$ is defined by
\[
H_{N}(t)=\int_{G_{N}}f_{N}(t,\mathbf{x})\log f_{N}(t,\mathbf{x})d\mathbf{x}\;.
\]

\begin{prop}
\label{p33}We have that $\frac{d}{dt}H_{N}(t)\le-4\mathcal{D}_{N}(f_{N}(t,\cdot))$.\end{prop}
\begin{proof}
Since the function $f_{N}$ solves the equation $\partial_{t}f_{N}=(1/2)\Delta f_{N}$
on $G_{N}$, and satisfies the boundary condition $\mathfrak{\mathscr{B}}_{ij}f_{N}\equiv0$
on $F_{ij}$ for all $i\neq j$, we obtain 
\begin{equation}
\frac{d}{dt}H_{N}(t)=\int_{G_{N}}\partial_{t}f_{N}(t,\mathbf{x})\log f_{N}(t,\mathbf{x})d\mathbf{x}=\int_{G_{N}}\Delta f_{N}(t,\mathbf{x})\log f_{N}(t,\mathbf{x})d\mathbf{x}\;.\label{hr}
\end{equation}
Hence, by \eqref{gr} with $u(\mathbf{x})=\log f_{N}(t,\mathbf{x})$
and $\mathbf{V}(\mathbf{x})=\nabla f_{N}(t,\mathbf{x})$, 
\begin{eqnarray*}
\frac{d}{dt}H_{N}(t) & = & -\sum_{i=1}^{N}\int_{G_{N}}\frac{\sigma_{c(i)}^{2}}{2}\frac{|\partial_{i}f_{N}(t,\mathbf{x})|^{2}}{f_{N}(t,\mathbf{x})}d\mathbf{x}\\
 &  & -\sum_{i\neq j}\frac{\lambda N\sigma_{c(i)}^{2}\sigma_{c(j)}^{2}}{2}\int_{F_{ij}}\left((f_{N})_{ij}-(f_{N})_{ji}\right)\left(\log(f_{N})_{ij}-\log(f_{N})_{ji}\right)(t,\mathbf{x})dS_{ij}(\mathbf{x})\;.
\end{eqnarray*}
The proof is completed by the elementary inequality $(a-b)(\log a-\log b)\ge4(\sqrt{a}-\sqrt{b})^{2}$.
\end{proof}
For $t_{1}<t_{2}$, define $\overline{f}_{N}^{[t_{1},t_{2}]}(\mathbf{x})=(t_{2}-t_{1})^{-1}\int_{t_{1}}^{t_{2}}f_{N}(t,\mathbf{x})dt$,
and denote $\overline{f}_{N}^{[0,T]}$ simply by $\overline{f}_{N}$.
\begin{cor}
\label{c34}For all $0\le t_{1}<t_{2}\le T$, we have that $\mathcal{D}_{N}(\overline{f}_{N}^{[t_{1},t_{2}]})\le CN/(t_{2}-t_{1})$
where the constant $C$ does not depend on $t_{1},\,t_{2}$ and $N$.\end{cor}
\begin{proof}
By the convexity of the Dirichlet form,
\[
\mathcal{D}_{N}(\overline{f}_{N}^{[t_{1},t_{2}]})\le\frac{1}{t_{2}-t_{1}}\int_{t_{1}}^{t_{2}}\mathcal{D}_{N}(f_{N}(t,\cdot))dt\;.
\]
By Proposition \ref{p33} and Assumption \ref{a3}, the right hand
side of the above inequality is bounded above by
\[
-\frac{1}{4(t_{2}-t_{1})}\int_{t_{1}}^{t_{2}}\frac{d}{dt}H_{N}(t)dt=\frac{H_{N}(t_{1})-H_{N}(t_{2})}{4(t_{2}-t_{1})}\le\frac{H_{N}(0)}{4(t_{2}-t_{1})}\le\frac{CN}{t_{2}-t_{1}}\;.
\]

\end{proof}

\subsection{\label{s312}Tightness}

Denote by $\mathbb{P}_{N}^{eq}$ the equilibrium process starting
from the invariant measure $d\mathbf{x}$. We recall the following
Dirichlet form for the one-component system \cite{S} with the interaction
parameter $\lambda>0$:
\begin{equation}
\widetilde{\mathcal{D}}_{N}(f)=\frac{1}{8}\sum_{i=1}^{N}\int_{G_{N}}\frac{(\partial_{i}f)^{2}}{f}d\mathbf{x}+\frac{\lambda N}{2}\sum_{i\neq j}\int_{F_{ij}}\left(\sqrt{f_{ij}}-\sqrt{f_{ji}}\right)^{2}dS_{ij}(\mathbf{x})\;.\label{e36}
\end{equation}
Note that this Dirichlet form is equivalent to that in our study in
the sense that 
\begin{equation}
C_{1}\mathcal{D}_{N}(f)\le\widetilde{\mathcal{D}}_{N}(f)\le C_{2}\mathcal{D}_{N}(f)\label{e35}
\end{equation}
where $C_{1}>0$ and $C_{2}>0$ are constants that depend only on
$\sigma_{1}$ and $\sigma_{2}$. In \cite[Proposition 3.4]{S}, the
exponential tightness of the one-component LIBM in equilibrium has
been developed, and the proof therein is entirely based on the estimates
in terms of the Dirichlet form $\widetilde{\mathcal{D}}_{N}(\cdot)$.
Accordingly, all the arguments are still valid for our model through
equivalence (\ref{e35}). In this manner, we obtain the following
tightness result for the equilibrium processes.
\begin{prop}
\label{p1}For all $\epsilon,\,\alpha>0$, it holds
\[
\limsup_{\delta\rightarrow0}\limsup_{N\rightarrow\infty}\frac{1}{N}\log\mathbb{P}_{N}^{eq}\bigg[\bigg|\bigg\{ i:\sup_{\substack{0\le s,\,t\le T\\
|s-t|\le\delta
}
}\left|x_{i}^{N}(t)-x_{i}^{N}(s)\right|\ge\epsilon\bigg\}\bigg|\ge N\alpha\bigg]=-\infty\;.
\]

\end{prop}
This super-exponential estimate for the equilibrium processes allows
us to develop the tightness of the non-equilibrium processes $\mathbb{P}_{N}$,
as follows.
\begin{cor}
\label{c36}For all $\epsilon,\,\alpha>0$,
\begin{equation}
\limsup_{\delta\rightarrow0}\limsup_{N\rightarrow\infty}\mathbb{P}_{N}\bigg[\bigg|\bigg\{ i:\sup_{\substack{0\le s,\,t\le T\\
|s-t|\le\delta
}
}\left|x_{i}^{N}(t)-x_{i}^{N}(s)\right|\ge\epsilon\bigg\}\bigg|\ge N\alpha\bigg]=0\;.\label{e37}
\end{equation}
In particular, $\{\mathbb{Q}_{N}\}_{N=1}^{\infty}$ is a tight family
in $C([0,T],\mathscr{M}(\mathbb{T})^{2})$.\end{cor}
\begin{proof}
Denote by $E_{N,\delta,\epsilon,\alpha}$ the event inside the bracket
of \eqref{e37}. By \cite[Proposition 8.2 of Appendix 1]{KL} and
by Assumption \ref{a3},
\begin{equation}
\mathbb{P}_{N}[E_{N,\delta,\epsilon,\alpha}]\le\frac{2+H[\mathbb{P}_{N}|\mathbb{P}_{N}^{eq}]}{\log\left(1+\mathbb{P}_{N}^{eq}[E_{N,\delta,\epsilon,\alpha}]^{-1}\right)}\le\frac{2+CN}{-\log\mathbb{P}_{N}^{eq}[E_{N,\delta,\epsilon,\alpha}]}\;.\label{e38}
\end{equation}
Hence, \eqref{e37} is a consequence of Proposition \ref{p1}.

We next prove the tightness of $\{\mathbb{Q}_{N}\}_{N=1}^{\infty}$.
It suffices to demonstrate that, for $\epsilon>0$ and $c\in\{1,\,2\}$,
\[
\lim_{\delta\rightarrow0}\limsup_{N\rightarrow\infty}\mathbb{P}_{N}\bigg[\sup_{\substack{0\le s,t\le T,\,|s-t|\le\delta}
}\frac{1}{N}\sum_{i\in\mathcal{T}_{c}^{N}}\left|x_{i}^{N}(t)-x_{i}^{N}(s)\right|\ge\epsilon\bigg]=0\;.
\]
This estimate follows from (\ref{e37}), since the last probability
is bounded above by
\[
\mathbb{P}_{N}\bigg[\bigg|\bigg\{ i:\sup_{\substack{0\le s,t\le T,\,|s-t|\le\delta}
}\left|x_{i}^{N}(t)-x_{i}^{N}(s)\right|\ge\epsilon/2\bigg\}\bigg|\ge N\epsilon/2\bigg]\;.
\]

\end{proof}

\subsection{Energy estimate}

We now establish an energy estimate. Let $\rho=\rho_{1}+\rho_{2}$. 
\begin{prop}
\label{p38}Suppose that $\mathbb{Q}_{\infty}$ is a weak limit of
$\{\mathbb{Q}_{N}\}_{N=1}^{\infty}$. Then, $\mathbb{Q}_{\infty}$
is concentrated on the trajectory of the form \textup{$(\rho_{1}(\cdot,x)dx,\,\rho_{2}(\cdot,x)dx)^{\dagger}$}
for some $\rho_{1},\,\rho_{2}\in L^{2}([0,T]\times\mathbb{T})$ which
are weakly differentiable in $x$, and satisfy
\begin{equation}
\int_{0}^{T}\int_{\mathbb{T}}\frac{\rho_{x}^{2}(t,x)}{\rho(t,x)}dxdt<\infty\;.\label{e315}
\end{equation}
\end{prop}
\begin{proof}
We claim first that all the conclusions of proposition follow from
the following estimate for all $\phi\in C^{0,1}([0,T]\times\mathbb{T})$:
\begin{equation}
\mathbb{E}_{N}\bigg[\int_{0}^{T}\frac{1}{N}\sum_{i=1}^{N}\phi'(x_{i}^{N}(t))dt\bigg]\le C\mathbb{E}_{N}\left[\int_{0}^{T}\frac{1}{N}\sum_{i=1}^{N}\phi^{2}(x_{i}^{N}(t))dt\right]^{1/2}\;.\label{e316}
\end{equation}
Indeed, it is standard, e.g., \cite[Section 5.7]{KL} that (\ref{e316})
implies that $\rho$ is not only absolutely continuous with respect
to the Lebesgue measure but also has a weak derivative $\rho_{x}$
that satisfies (\ref{e315}). By \cite[Lemma 4.3]{S} this energy
estimate implies $\rho\in L^{2}([0,T]\times\mathbb{T})$ and accordingly
$\rho_{1},\,\rho_{2}\in L^{2}([0,T]\times\mathbb{T})$.

We now prove \eqref{e316}. By \eqref{gr} with $u(\mathbf{x})=\phi(x_{1})+\cdots+\phi(x_{N})$
and $\mathbf{V}=(\overline{f}_{N},\cdots,\overline{f}_{N})$, the
left hand side of (\ref{e316}) becomes
\begin{equation}
\frac{T}{N}\int_{G_{N}}\sum_{i=1}^{N}\phi'(x_{i})\overline{f}_{N}(\mathbf{x})d\mathbf{x}=-\frac{T}{N}\int_{G_{N}}\sum_{i=1}^{N}\phi(x_{i})(\partial_{i}\overline{f}_{N})(\mathbf{x})d\mathbf{x}\;.\label{e317}
\end{equation}
By Cauchy-Schwarz's inequality, the right hand side of (\ref{e317})
is bounded above by
\begin{equation}
\frac{T}{N}\bigg[\int_{G_{N}}\sum_{i=1}^{N}\phi^{2}(x_{i})\overline{f}_{N}(\mathbf{x})d\mathbf{x}\int_{G_{N}}\sum_{i=1}^{N}\frac{(\partial_{i}\overline{f}_{N})^{2}(\mathbf{x})}{\overline{f}_{N}(\mathbf{x})}d\mathbf{x}\bigg]^{1/2}\;.\label{e318}
\end{equation}
Thus, (\ref{e316}) follows from Corollary \ref{c34}. 
\end{proof}

\subsection{\label{s313}Auxiliary martingales}

Recall the average density $\overline{\rho}_{c}=\int_{\mathbb{T}}\rho_{c}^{0}(x)dx$,
$c\in\{1,\,2\}$, and define a constant by
\[
\alpha=\bigg(\lambda+\frac{\overline{\rho}_{1}}{\sigma_{1}^{2}}+\frac{\overline{\rho}_{2}}{\sigma_{2}^{2}}\bigg)^{-1}\;.
\]
Define $\nu(x)=x$, $x\in[0,1)$ and regard $\nu(\cdot)$ as a discontinuous
function on $\mathbb{T}$. The process $z_{k}^{N}(t)$, $1\le k\le N$,
is defined by 
\begin{equation}
z_{k}^{N}(t)=x_{k}^{N}(t)+\frac{\alpha}{N}\sum_{i=1}^{N}\frac{1}{\sigma_{c(i)}^{2}}\nu(x_{i}^{N}(t)-x_{k}^{N}(t))\;.\label{e310}
\end{equation}
We claim that $z_{k}^{N}(t)-z_{k}^{N}(0)$ is a martingale with respect
to the filtration $\{\mathscr{F}_{t}:0\le t\le T\}$. To demonstrate
this, define a function $r_{N,k}:G_{N}\rightarrow\mathbb{R}$ by 
\begin{equation}
r_{N,k}(\mathbf{x})=x_{k}+\frac{\alpha}{N}\sum_{i:i\neq k}\frac{1}{\sigma_{c(i)}^{2}}\nu(x_{i}-x_{k})\;.\label{e301}
\end{equation}
Since $r_{N,k}$ is a linear function on $G_{N}$, it is obvious that
$\partial_{i}^{2}r_{N,k}\equiv0$ on $G_{N}$ for all $1\le i\le N$,
and it is also straightforward to check that $\mathscr{\mathscr{B}}_{ij}r_{N,k}\equiv0$
on $F_{ij}$ for all $1\le i\neq j\le N$. Hence, by (\ref{e3}),
$z_{k}^{N}(t)-z_{k}^{N}(0)=r_{N,k}(\mathbf{x}^{N}(t))-r_{N,k}(\mathbf{x}^{N}(0))$
is a martingale. Moreover, by (\ref{e4}), this martingale can be
written as
\begin{equation}
z_{k}^{N}(t)-z_{k}^{N}(0)=\alpha\lambda\sigma_{c(k)}\beta_{k}(t)+\frac{\alpha}{N}\sum_{i=1}^{N}\frac{1}{\sigma_{c(i)}}\beta_{i}(t)+\frac{\alpha}{N}\sum_{i:i\neq k}\frac{1}{\sigma_{c(i)}^{2}}\left[M_{ik}^{N}(t)-M_{ki}^{N}(t)\right]\;.\label{e311}
\end{equation}
For $1\le i\le N$ and $c\in\{1,2\}$, define the averaged local time
$A_{i,c}^{N}(t)$ by 
\begin{equation}
A_{i,c}^{N}(t)=\frac{1}{N}\sum_{j\in\mathcal{T}_{c}^{N}\setminus\{i\}}\left[A_{ij}^{N}(t)+A_{ji}^{N}(t)\right]\;.\label{e03}
\end{equation}
Then, the quadratic variation of the martingale $z_{k}^{N}(t)-z_{k}^{N}(0)$
can be written as 
\begin{equation}
\left\langle z_{k}^{N},z_{k}^{N}\right\rangle _{t}=\lambda\alpha^{2}\sigma_{c(k)}^{2}\bigg[\lambda t+\frac{1}{\sigma_{1}^{2}}A_{k,1}^{N}(t)+\frac{1}{\sigma_{2}^{2}}A_{k,2}^{N}(t)\bigg]+O(N^{-1})t\;.\label{e312}
\end{equation}

\subsection{\label{s314}Mollification of local times}

In the derivation of the hydrodynamic limit in the spirit of \cite{GPV,V2},
the evolution of the density field must be analyzed and the major
technical issue in this investigation is the approximation of the
interaction terms by a function of the empirical density. This step
is known as the replacement lemma, and we refer to \cite[Section 5]{KL}
for the detailed exposition. In the context of this work, this interaction
term is $J_{2}$ in \eqref{e321} and therefore, in view of \eqref{e312},
we must replace $A_{i,c}^{N}(t)$ with a function of the density field.
To this end, we introduce the \textit{local density} of particle configuration.
Fix $\epsilon>0$ and let $\iota_{\epsilon}=(2\epsilon)^{-1}\mathds{1}_{[-\epsilon,\epsilon]}$
be a function on $\mathbb{T}$. Then, for $\mathbf{x}\in G_{N}$,
we define the local density of type $c\in\{1,\,2\}$ around $x_{i}$,
$1\le i\le N$, by
\begin{equation}
\rho_{i,c}^{N,\epsilon}(\mathbf{x})=\frac{1}{N}\sum_{j\in\mathcal{T}_{c}^{N}}\iota_{\epsilon}(x_{j}-x_{i})\;.\label{e313}
\end{equation}
In view of \eqref{loc}, the local density $\rho_{i,c}^{N,\epsilon}(\mathbf{x})$
is the natural candidate for the required replacement, and the corresponding
approximation can be formally stated as follows. 
\begin{thm}
\label{t2}For all $c_{1},\,c_{2}\in\{1,\,2\}$, $\delta>0$, $0\le t_{1}<t_{2}\le T$,
and $h(\cdot,\,\cdot)\in C^{1}([0,T]\times\mathbb{T})$, we have that
\[
\limsup_{\epsilon\rightarrow0}\limsup_{N\rightarrow\infty}\mathbb{P}_{N}\bigg[\bigg|\frac{1}{N}\sum_{i\in\mathcal{T}_{c_{1}}^{N}}\int_{0}^{T}h(t,z_{i}^{N}(t))\left[dA_{i,c_{2}}^{N}(t)-\rho_{i,c_{2}}^{N,\epsilon}(\mathbf{x}^{N}(t))dt\right]\bigg|>\delta\bigg]=0\;.
\]

\end{thm}
We now prove this theorem by several steps. Define $\mathscr{P}_{N}=\{f\in\overline{\mathcal{C}}(G_{N}):\int_{G_{N}}f(x)dx=1\}$,
and for $\epsilon>0$ and $1\le i\le N$, define 
\[
M_{\epsilon,i}(\mathbf{x})=\sum_{j:j\neq i}\mathds{1}_{[-\epsilon,\epsilon]}(x_{j}-x_{i})\;.
\]
The following lemma was proven in \cite[Lemmata 2.6, 2.7 and 2.10]{S}. 
\begin{lem}
\label{lem37}For $f\in\mathscr{P}_{N}$ and $\epsilon\in(0,\,1/4)$,
the following estimates hold:
\begin{eqnarray*}
 &  & \frac{1}{N^{2}}\sum_{i=1}^{N}\int_{G_{N}}f(\mathbf{x})M_{\epsilon,i}(\mathbf{x})d\mathbf{x}\le C\epsilon\left[1+\left(\mathcal{\widetilde{D}}_{N}(f)/N\right)^{1/2}\right]\;,\\
 &  & \frac{1}{N^{2}}\sum_{i=1}^{N}\int_{G_{N}}\left|(\partial_{i}f)(\mathbf{x})\right|M_{\epsilon,i}(\mathbf{x})d\mathbf{x}\le C\epsilon^{1/2}\left[1+\left(\mathcal{\mathcal{\widetilde{D}}}_{N}(f)/N\right)^{3/4}\right]\;,\\
 &  & \frac{1}{N^{2}}\sum_{i\neq j}\int_{F_{ij}}\left|f_{ij}(\mathbf{x})-f_{ji}(\mathbf{x})\right|M_{\epsilon,i}(\mathbf{x})dS_{ij}(\mathbf{x})\le C(\epsilon^{1/4}+N^{-1/2})\left[1+\left(\mathcal{\mathcal{\widetilde{D}}}_{N}(f)/N\right)^{7/8}\right]\;.
\end{eqnarray*}
Furthermore, these estimates are still valid if we replace $\mathcal{\mathcal{\widetilde{D}}}_{N}$
with $\mathcal{D}_{N}$. 
\end{lem}
The proof of these estimates are entirely based on Green's formula
\eqref{gr}, and we refer to \cite[Section 2.2]{S} for the detailed
proofs of these estimates. Note that we can replace $\widetilde{\mathcal{D}}_{N}$
with $\mathcal{D}_{N}$ because of \eqref{e35}.

Let $z_{k}=r_{N,k}(\mathbf{x})$, $1\le k\le N$, where $r_{N,k}$
is defined in \eqref{e301}. The following lemma is a generalization
of \cite[Proposition 2.11]{S}.
\begin{lem}
\label{lem2} For all $f\in\mbox{\ensuremath{\mathscr{P}}}_{N}$,
$h\in C^{1}(\mathbb{T})$ and $0<\epsilon<\frac{1}{4}$, we have that
\begin{align}
 & \frac{1}{N^{2}}\sum_{i\neq j}\left|\int_{G_{N}}h(z_{i})f(\mathbf{x})\iota_{\epsilon}(x_{j}-x_{i})d\mathbf{x}-\frac{1}{2}\int_{F_{ij}}h(z_{i})(f_{ij}+f_{ji})(\mathbf{x})dS_{ij}(\mathbf{x})\right|\nonumber \\
 & \quad\le C\left(1+[\mathcal{D}_{N}(f)/N]^{\frac{7}{8}}\right)\left(\epsilon^{\frac{1}{4}}+N^{-\frac{1}{2}}\right)\label{e52}
\end{align}
for some constant $C=C(h)>0$ that only depends on $h$.\end{lem}
\begin{proof}
Define a function $g_{\epsilon}$ on $[0,\,1)$ by
\[
g_{\epsilon}(x)=\left(\frac{x}{2\epsilon}-\frac{1}{2}\right)\mathds{1}_{[0,\epsilon]}(x)+\left(\frac{x-1}{2\epsilon}+\frac{1}{2}\right)\mathds{1}_{[1-\epsilon,1)}(x)
\]
and regard $g_{\epsilon}$ as a discontinuous function on $\mathbb{T}$.
Fix $i\neq j$ and define a vector field $\mathbf{V}_{ij}(\mathbf{x})=h(z_{i})g_{\epsilon}(x_{i}-x_{j})\mathbf{e}_{j}$
which is continuous on $G_{N}$. Fix $f\in\mathscr{P}_{N}$. By \eqref{gr}
with $u=f$, and $\mathbf{V}=\mathbf{V}_{ij}$, we obtain 
\begin{equation}
\int_{G_{N}}(\partial_{j}f)(\mathbf{x})\left[h(z_{i})g_{\epsilon}(x_{i}-x_{j})\right]d\mathbf{x}=K_{ij}^{(1)}+K_{ij}^{(2)}\;,\label{ed1}
\end{equation}
where
\begin{align*}
K_{ij}^{(1)}= & \int_{G_{N}}f(\mathbf{x})\partial_{j}\left[h(z_{i})g_{\epsilon}(x_{i}-x_{j})\right]d\mathbf{x}\\
K_{ij}^{(2)}= & \int_{F_{ji}}f_{ji}(\mathbf{x})h(z_{i})g_{\epsilon}(x_{i}-x_{j})dS_{ji}(\mathbf{x})-\int_{F_{ij}}f_{ij}(\mathbf{x})h(z_{i})g_{\epsilon}(x_{i}-x_{j})dS_{ij}(\mathbf{x})\\
 & +\sum_{k:k\neq i,\,j}\left[\int_{F_{jk}}f_{jk}(\mathbf{x})h(z_{i})g_{\epsilon}(x_{i}-x_{j})dS_{jk}(\mathbf{x})-\int_{F_{kj}}f_{kj}(\mathbf{x})h(z_{i})g_{\epsilon}(x_{i}-x_{j})dS_{kj}(\mathbf{x})\right]\;.
\end{align*}
Note that $K_{ij}^{(1)}$ and $K_{ij}^{(2)}$ correspond to the first
and second terms of the right hand side of \eqref{gr}, respectively.
By simple computations, we deduce
\begin{equation}
K_{ij}^{(1)}=\int_{G_{N}}f(\mathbf{x})\bigg[\frac{\alpha}{N\sigma_{c(j)}^{2}}h'(z_{i})g_{\epsilon}(x_{i}-x_{j})+h(z_{i})\iota_{\epsilon}(x_{i}-x_{j})\bigg]d\mathbf{x}\;.\label{ed2}
\end{equation}
Since $g_{\epsilon}(x_{i}-x_{j})=\frac{1}{2}$ on $F_{ij}$ and $-\frac{1}{2}$
on $F_{ji}$, and since $h(z_{i})g_{\epsilon}(x_{i}-x_{j})$ has same
value on $F_{jk}$ and $F_{kj}$ for $k\neq i,\,j$, we can simplify
$K_{ij}^{(2)}$ to 
\begin{align}
K_{ij}^{(2)}= & -\frac{1}{2}\int_{F_{ij}}h(z_{i})\left(f_{ij}(\mathbf{x})+f_{ji}(\mathbf{x})\right)dS_{ij}(\mathbf{x})\nonumber \\
 & +\sum_{k:k\neq i,\,j}\left[\int_{F_{jk}}(f_{jk}(\mathbf{x})-f_{kj}(\mathbf{x}))h(z_{i})g_{\epsilon}(x_{i}-x_{j})dS_{jk}(\mathbf{x})\right]\;.\label{ed3}
\end{align}
By \eqref{ed1}, \eqref{ed2}, \eqref{ed3} and by an elementary inequality
$\left|g_{\epsilon}\right|\le\frac{1}{2}\chi_{\epsilon}$, we can
bound the left hand side of (\ref{e52}) by
\begin{align*}
 & \frac{C(h)}{N^{2}}\sum_{i=1}^{N}\int_{G_{N}}\left|\partial_{i}f(\mathbf{x})\right|M_{\epsilon,i}(\mathbf{x})d\mathbf{x}+\frac{C}{N^{3}}\sum_{i=1}^{N}\int_{G_{N}}f(\mathbf{x})M_{\epsilon,i}(\mathbf{x})d\mathbf{x}\\
 & +\frac{C(h)}{N^{2}}\sum_{i\neq j}\int_{F_{jk}}\left|f_{ij}(\mathbf{x})-f_{ij}(\mathbf{x})\right|M_{\epsilon,i}(\mathbf{x})dS_{ij}(\mathbf{x})\;.
\end{align*}
The proof is completed by Lemma \ref{lem37}. 
\end{proof}
Based on the previous lemma, we obtain the following super-exponential
estimate for the equilibrium process.
\begin{prop}
\label{t5}For all $c_{1},\,c_{2}\in\{1,\,2\}$, $\delta>0$, $0\le t_{1}<t_{2}\le T$,
and $h\in C^{1}(\mathbb{T})$, we have that 
\begin{align}
 & \limsup_{\epsilon\rightarrow0}\limsup_{N\rightarrow\infty}\frac{1}{N}\log\mathbb{P}_{N}^{eq}\bigg[\bigg|\frac{1}{N}\sum_{i\in\mathcal{T}_{c_{1}}^{N}}\int_{t_{1}}^{t_{2}}h(z_{i}^{N}(t))\left[dA_{i,c_{2}}^{N}(t)-\rho_{i,c_{2}}^{N,\epsilon}(\mathbf{x}^{N}(t))dt\right]\bigg|>\delta\bigg]=-\infty\;.\label{e56}
\end{align}
\end{prop}
\begin{proof}
Denote simply by $E(N,h,\epsilon,\delta)$ the event inside the bracket
of \eqref{e56}. By Chebyshev's inequality, for any $a>0$, 
\begin{align}
 & \frac{1}{N}\log\mathbb{P}_{N}^{eq}\left[E(N,h,\epsilon,\delta)\right]\nonumber \\
 & \le-a\delta+\frac{1}{N}\log\mathbb{E}_{N}^{eq}\exp\bigg\{ a\sum_{i\in\mathcal{T}_{c_{1}}^{N}}\left|\int_{t_{1}}^{t_{2}}h(z_{i}^{N}(t))\left[dA_{i,c_{2}}^{N}(t)-\rho_{i,c_{2}}^{N,\epsilon}(\mathbf{x}^{N}(t))dt\right]\right|\bigg\}\;.\label{ee1}
\end{align}
Let $\mathcal{E}_{N}=\{(e_{1},e_{2},\cdots,e_{N}):e_{i}=\pm1\,\,\forall\ensuremath{1\le i\le N}\}$.
By the inequality $e^{|x|}\le e^{x}+e^{-x}$ and Feynman-Kac's formula,
we can bound the last expectation by 
\begin{eqnarray}
 &  & \sum_{(e_{1},\cdots,e_{N})\in\mathcal{E}_{N}}\mathbb{E}_{N}^{eq}\exp\bigg\{ a\sum_{i\in\mathcal{T}_{c_{1}}^{N}}e_{i}\int_{t_{1}}^{t_{2}}h(z_{i}^{N}(t))\left[dA_{i,c_{2}}^{N}(t)-\rho_{i,c_{2}}^{N,\epsilon}(\mathbf{x}^{N}(t))dt\right]\bigg\}\nonumber \\
 &  & \quad\le\sum_{\mathbf{e}\in\mathcal{E}_{N}}\exp\{(t_{2}-t_{1})\lambda_{h,N,\epsilon,a,\mathbf{e}}\}\;,\label{ee2}
\end{eqnarray}
where $\lambda_{h,N,\epsilon,a,\mathbf{e}}$, $\mathbf{e}=(e_{1},\cdots,e_{N})\in\mathcal{E}_{N}$,
is the largest eigenvalue of the operator
\[
\mathscr{L}_{N}+\frac{a}{N}\sum_{i\in\mathcal{T}_{c_{1}}^{N},\,j\in\mathcal{T}_{c_{2}}^{N}}e_{i}\,h(z_{i})\bigg\{\iota_{\epsilon}(x_{j}-x_{i})-\delta(x_{j}-x_{i})\bigg\}\;.
\]
Assume now that, for all $\mathbf{e}\in\mathcal{E}_{N}$,
\begin{equation}
\lambda_{h,N,\epsilon,a,\mathbf{e}}\le CN\left[a\left(\epsilon^{\frac{1}{4}}+N^{-\frac{1}{2}}\right)+a^{8}\left(\epsilon^{\frac{1}{4}}+N^{-\frac{1}{2}}\right)^{8}\right]\label{ee3}
\end{equation}
for some constant $C$ which depends only on $h$. Then, we can deduce
from \eqref{ee1}, \eqref{ee2} and \eqref{ee3} that
\[
\frac{1}{N}\log\mathbb{P}_{N}^{eq}\left[E(N,h,\epsilon,\delta)\right]\le-a\delta+\log2+C\left[a\left(\epsilon^{\frac{1}{4}}+N^{-\frac{1}{2}}\right)+a^{8}\left(\epsilon^{\frac{1}{4}}+N^{-\frac{1}{2}}\right)^{8}\right]
\]
and therefore the left hand side of \eqref{e56} is bounded above
by $-a\delta+\log2$. This completes the proof since $a$ is an arbitrary
positive number. Thus, the proof of proposition is reduced to the
verification of \eqref{ee3}. To this end, recall (cf. \cite[Section 7]{KL})
that the variational formula for $\lambda_{h,N,\epsilon,a,\mathbf{e}}$
is $\sup_{f\in\mathscr{P}_{N}}$ of 
\begin{align*}
\frac{a}{N}\sum_{i\in\mathcal{T}_{c_{1}}^{N},\,j\in\mathcal{T}_{c_{2}}^{N}}e_{i}\Biggl[ & \int_{G_{N}}h(z_{i})f(\mathbf{x})\iota_{\epsilon}(x_{j}-x_{i})dx-\frac{1}{2}\int_{F_{ij}}h(z_{i})\left(f_{ij}(\mathbf{x})+f_{ji}(\mathbf{x})\right)dS_{ij}(\mathbf{x})\Biggr]-\mathcal{D}_{N}(f)\;.
\end{align*}
By Lemma \ref{lem2}, we can bound this expression above by 
\[
N\left[Ca\left(1+[\mathcal{D}_{N}(f)/N]^{\frac{7}{8}}\right)\left(\epsilon^{\frac{1}{4}}+N^{-\frac{1}{2}}\right)-\mathcal{D}_{N}(f)/N\right]\;.
\]
Hence, we can prove \eqref{ee3} by an elementary inequality $a^{7/8}b-a\le(7b/8)^{8}$
for $a,\,b>0$. 
\end{proof}
By the argument presented in \eqref{e38}, we obtain the following
corollary. Notice that the test function $h$ depends only on the
spatial variable.
\begin{cor}
\label{p39}For all $c_{1},\,c_{2}\in\{1,\,2\}$, $\delta>0$, $0\le t_{1}<t_{2}\le T$,
and $h\in C^{1}(\mathbb{T})$, 
\[
\limsup_{\epsilon\rightarrow0}\limsup_{N\rightarrow\infty}\mathbb{P}_{N}\bigg[\bigg|\frac{1}{N}\sum_{i\in\mathcal{T}_{c_{1}}^{N}}\int_{t_{1}}^{t_{2}}h(z_{i}^{N}(t))\left[dA_{i,c_{2}}^{N}(t)-\rho_{i,c_{2}}^{N,\epsilon}(\mathbf{x}^{N}(t))dt\right]\bigg|>\delta\bigg]=0\;.
\]

\end{cor}
We now prove Theorem \ref{t2}.
\begin{proof}[Proof of Theorem \ref{t2}]
For any $\eta>0$, we can find a partition $0=t_{0}<t_{1}<\cdots<t_{M+1}=T$
of $[0,\,T]$ so that the function $\widetilde{h}(t,x)=\sum_{i=0}^{M}\mathds{1}_{[t_{i},t_{i+1}]}(t)h(t_{i},x)$
satisfies $||\widetilde{h}-h||_{\infty}<\eta$. By Corollary \ref{p39},
the statement of theorem holds for $\widetilde{h}$. Since $\eta>0$
is arbitrary, it suffices to verify that 
\begin{eqnarray}
 &  & \limsup_{M\rightarrow\infty}\limsup_{N\rightarrow\infty}\mathbb{P}_{N}\bigg[\bigg|\frac{1}{N}\sum_{i\in\mathcal{T}_{c_{1}}^{N}}A_{i,c_{2}}^{N}(T)\bigg|>M\bigg]=0\;,\;\label{ep1}\\
 &  & \limsup_{M\rightarrow\infty}\limsup_{\epsilon\rightarrow0}\limsup_{N\rightarrow\infty}\mathbb{P}_{N}\bigg[\bigg|\frac{1}{N}\sum_{i\in\mathcal{T}_{c_{1}}^{N}}\int_{0}^{T}\rho_{i,c_{2}}^{N,\epsilon}(\mathbf{x}^{N}(t))dt\bigg|>M\bigg]=0\;.\label{ep2}
\end{eqnarray}
For \eqref{ep1}, it is enough to prove that
\begin{equation}
\mathbb{E}_{N}\left[\frac{1}{N^{2}}\sum_{i\neq j}A_{ij}^{N}(t)\right]\le C\label{ep3}
\end{equation}
for some constant $C$ depending only on $T$. Note that the last
expectation can be written as 
\begin{equation}
\frac{t}{N^{2}}\sum_{i\neq j}\int_{F_{ij}}\left(\bar{f}_{N}^{[0,t]}\right)_{ij}(\mathbf{x})dS_{ij}(\mathbf{x})\;.\label{e323-1}
\end{equation}
By \cite[Lemma 2.5]{S}, this term is bounded by $\big[2+\sqrt{8N^{-1}\mathcal{\widetilde{D}}_{N}(\bar{f}_{N}^{[0,t]})}\big]t$.
Hence the proof of \eqref{ep3} can be completed by (\ref{e35}) and
Corollary \ref{c34}. 

For \eqref{ep2}, observe first that
\begin{eqnarray*}
 &  & \mathbb{P}_{N}\bigg[\bigg|\frac{1}{N}\sum_{i\in\mathcal{T}_{c_{1}}^{N}}\int_{0}^{T}\rho_{i,c_{2}}^{N,\epsilon}(\mathbf{x}^{N}(t))dt\bigg|>M\bigg]\\
 &  & \le\mathbb{P}_{N}\bigg[\bigg|\frac{1}{N}\sum_{i\in\mathcal{T}_{c_{1}}^{N}}A_{i,c_{2}}^{N}(T)\bigg|>\frac{M}{2}\bigg]+\mathbb{P}_{N}\bigg[\bigg|\frac{1}{N}\sum_{i\in\mathcal{T}_{c_{1}}^{N}}\left(A_{i,c_{2}}^{N}(T)-\int_{0}^{T}\rho_{i,c_{2}}^{N,\epsilon}(\mathbf{x}^{N}(t))dt\right)\bigg|>\frac{M}{2}\bigg]\;.
\end{eqnarray*}
Then, two probabilities can be controlled respectively by \eqref{ep1},
and by Corollary \ref{p39} with $h\equiv1$, $t_{1}=0$ and $t_{2}=T$,
respectively. 
\end{proof}

\subsection{Proof of Theorem \ref{t1}\label{s322}}

In the one-component system, the limit of uncolored empirical density
is obtained by the solution of the heat equation. By similar computation,
we can derive the following lemma.
\begin{lem}
\label{lem1}Let $\mathbb{Q}_{\infty}$ be a weak limit of $\{\mathbb{Q}_{N}\}_{N=1}^{\infty}$.
Then, $\mathbb{Q}_{\infty}$ is concentrated on the trajectory of
the form $(\rho_{1}(\cdot,x)dx,\,\rho_{2}(\cdot,x)dx)^{\dagger}$,
where $(\rho_{1},\rho_{2})$ weakly satisfies
\begin{equation}
\frac{\partial}{\partial t}\bigg[\frac{\rho_{1}(t,x)}{\sigma_{1}^{2}}+\frac{\rho_{2}(t,x)}{\sigma_{2}^{2}}\bigg]=\frac{1}{2}\rho_{xx}(t,x)\;.\label{e319}
\end{equation}
\end{lem}
\begin{proof}
For any $f\in C^{1,2}([0,T]\times\mathbb{T})$, by \eqref{e3} and \eqref{e4},
we can check that 
\begin{equation}
\frac{1}{N}\sum_{i=1}^{N}\bigg[\frac{1}{\sigma_{c(i)}^{2}}f(t,x_{i}^{N}(t))-\frac{1}{\sigma_{c(i)}^{2}}f(0,x_{i}^{N}(0))-\int_{0}^{t}\left(f_{t}+\frac{1}{2}f_{xx}\right)(s,x_{k}^{N}(s))ds\bigg]\label{e344}
\end{equation}
is a martingale, and can be expressed as
\[
\frac{1}{N}\sum_{i=1}^{N}\frac{1}{\sigma_{c(i)}}\int_{0}^{t}f_{x}(s,x_{i}^{N}(s))d\beta_{i}(s):=M_{f}(t)
\]
Since $\mathbb{E}_{N}[M_{f}^{2}(t)]=O(N^{-1})$, the expression \eqref{e344}
converges to $0$ in probability as $N\rightarrow\infty$. This completes
the proof. 
\end{proof}

\begin{proof}[Proof of Theorem \ref{t1}]
 We start by applying Ito's formula to the density field of $\left\{ z_{i}^{N}(\cdot)\right\} _{i\in\mathcal{T}_{1}^{N}}$.
For $f\in C^{1,2}([0,T]\times\mathbb{T})$, we can write 
\begin{equation}
\frac{1}{N}\sum_{i\in\mathcal{T}_{1}^{N}}\left[f(t,z_{i}^{N}(t))-f(0,\,z_{i}^{N}(0))\right]=J_{1}(t)+J_{2}(t)+J_{3}(t)\label{e321}
\end{equation}
where
\begin{align*}
J_{1}(t) & =\frac{1}{N}\sum_{i\in\mathcal{T}_{1}^{N}}\int_{0}^{t}f_{t}(s,z_{i}^{N}(s))ds\;,\;\;J_{2}(t)=\frac{1}{2N}\sum_{i\in\mathcal{T}_{1}^{N}}\int_{0}^{t}f_{xx}(s,z_{i}^{N}(s))d\left\langle z_{i}^{N},z_{i}^{N}\right\rangle _{s}\;,\\
J_{3}(t) & =\frac{1}{N}\sum_{i\in\mathcal{T}_{1}^{N}}\int_{0}^{t}f_{x}(s,z_{i}^{N}(s))dz_{i}^{N}(s)\;.
\end{align*}
We first demonstrate that the martingale $J_{3}(t)$ is negligible.
By (\ref{e311}), (\ref{e312}) and \eqref{ep3}, 
\begin{equation}
\mathbb{E}_{N}[J_{3}^{2}(t)]\le CN^{-1}t+CN^{-1}\mathbb{E}_{N}\bigg[N^{-2}\sum_{i\neq j}A_{ij}^{N}(t)\bigg]=O(N^{-1})\;.\label{e322}
\end{equation}
Thus, $J_{3}(t)$ is negligible by Doob's inequality.

We now consider the quadratic variation part in $J_{2}(t)$. By (\ref{e312})
and Theorem \ref{t2}, we are able to approximate $d\left\langle z_{i}^{N},z_{i}^{N}\right\rangle _{s}$
by 
\begin{equation}
\lambda\alpha^{2}\sigma_{c(i)}^{2}\left[\lambda+\frac{\mu_{1}^{N}(s)*\iota_{\epsilon}(x_{i}^{N}(s))}{\sigma_{1}^{2}}+\frac{\mu_{2}^{N}(s)*\iota_{\epsilon}(x_{i}^{N}(s))}{\sigma_{2}^{2}}\right]ds\;.\label{e324}
\end{equation}
For $\rho_{1},\,\rho_{2}\in L^{2}(\mathbb{T})$, define 
\[
F(t,x)=F_{\rho_{1},\rho_{2}}(t,x)=x+\alpha\int_{\mathbb{T}}\nu(y-x)\left(\rho_{1}/\sigma_{1}^{2}+\rho_{2}/\sigma_{2}^{2}\right)(t,y)dy\;.
\]
Note that we can rewrite (\ref{e310}) as 
\begin{equation}
z_{i}^{N}(t)=x_{i}^{N}(t)+(\alpha/\sigma_{1}^{2})\left\langle \mu_{1}^{N}(t),\nu(\cdot-x_{i}^{N}(t))\right\rangle +(\alpha/\sigma_{2}^{2})\left\langle \mu_{2}^{N}(t),\nu(\cdot-x_{i}^{N}(t))\right\rangle \;.\label{e325}
\end{equation}
If we replace $\mu_{c}^{N}(t)$, $c\in\{1,\,2\}$, with $\rho_{c}(t,x)dx$,
then the right hand side becomes $F(t,x_{i}^{N}(t))$. By combining
this observation, (\ref{e321}), (\ref{e322}), (\ref{e324}), (\ref{e325})
and Proposition \ref{p38}, we can conclude that, for any $\delta>0$,
\begin{align}
\limsup_{\epsilon\rightarrow0}\mathbb{Q}_{\infty} & \bigg[\bigg|(\rho_{1}(\cdot,x)dx,\,\rho_{2}(\cdot,x)dx)^{\dagger}:\nonumber \\
 & \int_{\mathbb{T}}f(T,\,F(T,x))\rho_{1}(T,x)dx-\int_{\mathbb{T}}f(0,F(0,x))\rho_{1}(0,x)dx\nonumber \\
 & -\int_{0}^{T}\int_{\mathbb{T}}\bigg[f_{t}+\frac{\lambda\alpha^{2}\sigma_{1}^{2}}{2}\bigg(\lambda+\frac{\rho_{1}*\iota_{\epsilon}}{\sigma_{1}^{2}}+\frac{\rho_{2}*\iota_{\epsilon}}{\sigma_{2}^{2}}\bigg)f_{xx}\bigg]\rho_{1}(s,x)dxds\bigg|>\delta\bigg]=0\;,\label{e326}
\end{align}
where $f_{t}$ and $f_{xx}$ are evaluated at $(s,F(s,x))$, while
$\rho_{c}*\iota_{\epsilon}=\rho_{c}(s,\cdot)*\iota_{\epsilon}$ is
evaluated at $(s,x)$. By Proposition \ref{p38}, we know that $\rho_{c}*\iota_{\epsilon}$
converges to $\rho_{c}$ in $L^{2}$ as $\epsilon\rightarrow0$ and
therefore, we obtain from (\ref{e326}) that 
\begin{align}
\mathbb{Q}_{\infty}\bigg[ & (\rho_{1}(\cdot,x)dx,\,\rho_{2}(\cdot,x)dx)^{\dagger}:\nonumber \\
 & \int_{\mathbb{T}}f(T,\,F(T,x))\rho_{1}(T,x)dx-\int_{\mathbb{T}}(0,F(0,x))\rho_{1}(0,x)dx\nonumber \\
 & -\int_{0}^{T}\int_{\mathbb{T}}\bigg[f_{t}+\frac{\lambda\alpha^{2}\sigma_{1}^{2}}{2}\bigg(\lambda+\frac{\rho_{1}}{\sigma_{1}^{2}}+\frac{\rho_{2}}{\sigma_{2}^{2}}\bigg)f_{xx}\bigg]\rho_{1}dxds=0\bigg]=0\;.\label{e327}
\end{align}
As before, $f_{t}$ and $f_{xx}$ are evaluated at $(s,F(s,x))$,
and $\rho_{1}$ and $\rho_{2}$ are evaluated at $(s,x)$. Now we
wish to replace $f(t,F(t,x))$ with $g(t,x)\in C^{1,2}([0,\,T]\times\mathbb{T})$
to complete the derivation of equation \eqref{em1}. To this end,
first observe that, for any $h\in L^{1}(\mathbb{T})$, 
\[
\frac{d}{dx}\left[\int_{\mathbb{T}}\nu(y-x)h(y)dy\right]=h(x)-\int_{\mathbb{T}}h(y)dy\;.
\]
Hence, we can write $F_{x}(t,x)=\alpha U(t,x)$ where 
\[
U(t,x)=\lambda+\rho_{1}(t,x)/\sigma_{1}^{2}+\rho_{2}(t,x)/\sigma_{2}^{2}>\lambda>0\;.
\]
Thus, there exists the inverse function $G(t,\cdot)$ of $F(t,\cdot)$
so that $F(t,G(t,x))=G(t,F(t,x))=x$ for all $t$ and $x$. Now we
can insert $f(t,x)=g(t,G(t,x))$ which implies $f(t,F(t,x))=g(t,x)$. 

At this moment, we need that $g\in C^{1,2}$, while our computations
leading to (\ref{e327}) requires $f\in C^{1,2}$, and therefore we
should have $f,\,g\in C^{1,2}$ simultaneously. This is guaranteed
if $\rho_{1},\,\rho_{2}\in C^{1,2}$. Hence, we first assume that
$\rho_{1},\,\rho_{2}$ are smooth $x$. For this case, we can check
$F_{xx}=\alpha U_{x}$ and $F_{t}=\frac{\alpha}{2}\rho_{x}$ where
the latter follows from Lemma \ref{lem1}. Hence, by elementary computations
we can check 
\[
G_{x}(t,x)=\frac{1}{\alpha U(t,F(t,x))}\;,\;\;G_{xx}(t,x)=-\frac{U_{x}(t,F(t,x))}{\alpha^{2}U^{3}(t,F(t,x))}\mbox{\;and}\;G_{t}(t,x)=-\frac{\rho_{x}(t,F(t,x))}{2U(t,F(t,x))}\;.
\]
Finally, substitute $f(t,x)$ in (\ref{e327}) by $g(t,G(t,\,x))$
to obtain 
\begin{align}
\mathbb{Q}_{\infty}\bigg[ & (\rho_{1}(\cdot,x)dx,\,\rho_{2}(\cdot,x)dx)^{\dagger}:\nonumber \\
 & \int_{\mathbb{T}}g(T,x)\rho_{1}(T,x)dx-\int_{\mathbb{T}}f(0,x)\rho_{1}(0,\,x)dx\nonumber \\
 & -\int_{0}^{T}\int_{\mathbb{T}}\bigg[g_{t}-\bigg(\frac{\rho_{x}}{2U}+\frac{\lambda\sigma_{1}^{2}U_{x}}{2U^{2}}\bigg)g_{x}+\frac{\lambda\sigma_{1}^{2}}{2U}g_{xx}\bigg]\rho_{1}(s,x)dxds=0\bigg]=1.\label{e328}
\end{align}
This completes the proof when $\rho_{1},\,\rho_{2}$ are smooth. For
the general case, we mollify $\rho_{1},\,\rho_{2}$ by $\rho_{1}*\phi_{\eta}$,
$\rho_{2}*\phi_{\eta}$ with the smooth mollifier $\{\phi_{\eta}\}_{\eta>0}$
and then apply the argument above to $\rho_{1}*\phi_{\eta}$, $\rho_{2}*\phi_{\eta}$.
Ultimately, we allow $\eta\rightarrow0$ to obtain (\ref{e328}).
The precise argument of this procedure can be found in \cite[Lemma 4.12]{S}. 

In particular, (\ref{e328}) proves the first coordinate of \eqref{em1},
i.e., the equation for $\rho_{1}$. The equation for $\rho_{2}$ can
be proven in an identical manner. Therefore, the identification of
limit points as the weak solution of \eqref{em1} with initial condition
$\widetilde{\rho}^{\,0}(x)$ is completed. 

We postpone the discussion of the uniqueness of weak solution to the
next section. 
\end{proof}

\section{Uniqueness\label{s4}}

\subsection{Revisit: two-color system\label{s41}}

We start by considering the two-color system. At the time when this
paper was written, there existed three limit theorems for the two-color
system: SSEP \cite{Q}, ZRP \cite{GJL} and LIBM \cite{S}. These
three models share the same form of hydrodynamic limit. The evolution
of the limiting particle densities $\rho_{1},\,\rho_{2}$ of the two
colors are obtained as the solution of the following PDE:
\begin{equation}
\frac{\partial}{\partial t}\begin{pmatrix}\rho_{1}\\
\rho_{2}
\end{pmatrix}=\frac{1}{2}\nabla\cdot\left[\begin{pmatrix}\frac{\rho_{1}}{\rho}D(\rho)+\frac{\rho_{2}}{\rho}S(\rho) & \frac{\rho_{1}}{\rho}(D(\rho)-S(\rho))\\
\frac{\rho_{2}}{\rho}(D(\rho)-S(\rho)) & \frac{\rho_{2}}{\rho}D(\rho)+\frac{\rho_{1}}{\rho}S(\rho)
\end{pmatrix}\nabla\begin{pmatrix}\rho_{1}\\
\rho_{2}
\end{pmatrix}\right]\;.\label{e41}
\end{equation}
We briefly explain this equation:
\begin{itemize}
\item $\rho=\rho_{1}+\rho_{2}$ is the uncolored limiting density
\item $D(\rho)$ is the bulk-diffusion coefficient, $i.e.,$ $\rho$ is
the solution of
\begin{equation}
\frac{\partial\rho}{\partial t}=\frac{1}{2}\nabla\cdot[D(\rho)\nabla\rho].\label{e42}
\end{equation}
We can also derive this equation from (\ref{e41}) by simply adding
two equations in (\ref{e41}). For instance, for the SSEP and the
LIBM, $D(\rho)\equiv1$, so that equation \eqref{e42} becomes the
heat equation. That is because the nature of interaction is reflection.
For the ZRP, $D(\rho)$ is not a constant function and we refer to
\cite[Section 5]{KL} for details. 
\item $S(\rho)$ is the self-diffusion coefficient in the equilibrium with
density $\rho$. The closed form of the self-diffusion coefficient
is known only for LIBM \cite{G} and ZRP \cite{GJL}. In particular,
for the LIBM, $S(\rho)=\frac{\lambda}{\lambda+\rho}$ where $\lambda$
is the interaction parameter. The closed form is not known for the
SSEP but the regularity of $S(\cdot)$ has been established in \cite{LOV}. 
\end{itemize}
This kind of universality is an interesting feature of the theory
of interacting particle systems. In particular, we can derive the
non-equilibrium behavior of the tagged particles, the so-called propagation
of chaos \cite{R}, from the limit theorem for two-color system. We
emphasize here that the uniqueness of equation (\ref{e41}) is not
a significant issue. This becomes obvious when we substitute $\rho_{2}=\rho-\rho_{1}$
in the first equation of (\ref{e41}) to obtain
\begin{equation}
\frac{\partial\rho_{1}}{\partial t}=\frac{1}{2}\nabla\cdot\left[S(\rho)\nabla\rho_{1}+\frac{(D(\rho)-S(\rho))\nabla\rho}{\rho}\rho_{1}\right].\label{e43}
\end{equation}
Given that $\rho$ is the solution of the master equation (\ref{e42}),
we can simply regard (\ref{e43}) as a linear parabolic PDE. Thus,
under appropriate initial conditions and the non-degeneracy of $S(\cdot)$,
the uniqueness is automatically guaranteed. 
\begin{rem}
We can observe an interesting property of the particle system from
(\ref{e42}) and (\ref{e43}); \textit{the bulk evolution of particles
of specific color is not governed by the bulk-diffusion coefficient
$D(\cdot)$ but by the self-diffusion coefficient $S(\cdot)$ only. }
\end{rem}

\begin{rem}
An $m$-color system with $m\ge3$ produce exactly the same result;
(\ref{e43}) does not depend on the number of colors we used. 
\end{rem}
Now we consider  equation (\ref{em1}) for the two-component system.
In this case, due to the inhomogeneity of diffusivity of particles,
the master equation (\ref{e319}), which corresponds to \eqref{e42}
for homogeneous system, cannot be solved by itself. Accordingly, an
iterative strategy for solving the two-color system is unavailable
and instead, we have to confront the system of equation (\ref{em1})
in a direct manner. Remark that the quasi-linear parabolic equation
of the form (\ref{em1}) is known as a \textit{cross-diffusion equation. }

\subsection{Uniqueness }

The general theory for a cross-diffusion equation was thoroughly explained
in \cite{A}. In this section, we use this general theory to develop
the local uniqueness of (\ref{em1}).

Let us consider an $n$-dimensional quasi-linear equation 
\begin{equation}
\partial_{t}\mathbf{u}=\mbox{div}\cdot\left[A(\mathbf{u})\nabla\mathbf{u}\right]\label{e44}
\end{equation}
where $\mathbf{u}(\cdot)$ is an $n$-dimensional vector function
and $A(\cdot)$ is an $n\times n$ matrix functional. In most physical
situations of multi-component diffusive flow (\textit{e.g.,} \cite{A,BGG,G,JS,SKT}
and models therein) the diffusion matrix $A(\mathbf{u})$ is neither
symmetric nor positive semi-definite. Instead, it becomes evident
that the physically relevant condition for $A(\mathbf{u})$ is the
\textit{normal ellipticity}. More precisely, a square matrix $M$
is called normally elliptic\textit{ }if all of its eigenvalues have
a positive real part. Equation (\ref{e44}) is considered normally
elliptic if $A(\mathbf{u})$ is normally elliptic for all $\mathbf{u}$.
The normally elliptic parabolic equations differ intrinsically from
the uniformly elliptic one. In particular, in \cite{SJ}, examples
of normally elliptic equations without the maximum principle or even
worse than that, equations that \textit{blow up} in finite time, were
suggested. In experimental physics (cf. \cite{DT}), this behavior
of multi-component system has also been verified. 

The normally elliptic equation has been analyzed in \cite{A}, in
which the local existence and the uniqueness of solution were obtained.
In particular, the following theorem is a direct consequence of \cite[Theorems 14.4, 14.6 and 15.1]{A}.
\begin{thm}
\label{t3}Suppose that a function $\widetilde{\rho}^{\,0}$ satisfies
Assumption \ref{a1}. Then, there exists $t^{+}(\widetilde{\rho}^{\,0})>0$
such that equation (\ref{em1}) has a unique weak non-negative solution
$(\rho_{1},\rho_{2})^{\dagger}$ provided that $T<t^{+}(\widetilde{\rho}^{\,0})$. \end{thm}
\begin{proof}
The normal ellipticity of $D(\rho_{1},\rho_{2})$ of (\ref{em2})
is obvious since both of its trace and determinant are positive. We
note here that the condition $\widetilde{\rho}^{\,0}\in W_{1,p}(\mathbb{T})^{2}$
for some $p>2$ of Assumption \ref{a1} is used here for satisfying
the requirement of \cite[Theorem 14.4]{A}. 
\end{proof}
This theorem proves part (1) of Theorem \ref{t1}. As we can anticipate
from the peculiar behavior of the solution of certain examples in
\cite{SJ}, the analysis of equation (\ref{e44}) is more delicate
than the analysis of usual parabolic equations. In particular, the
general theory established in \cite[Theorem 14.4]{A} requires a priori
bound of $\sup_{t\ge0}||\widetilde{\rho}(t,\,\cdot)||_{W_{k,p}}$
for some $k$ and $p$ to achieve the global uniqueness result. This
bound for our model would be very difficult to obtain. We also stress
here that the reference \cite{SJ} demonstrate a counter-example for
which the global uniqueness does not hold.

\subsection{The Maxwell--Stefan equation\label{s43}}

Equation (\ref{em1}) is not only an normally elliptic equation but
also has some underlying physical structures, which may hopefully
be exploited to develop a more concrete result than what we obtained
in the previous section. In fact, in addition to the normal ellipticity,
we observed that (\ref{em1}) is equivalent to the well-known Maxwell--Stefan
equation \cite{M}. We introduce the Maxwell--Stefan equation and
refer to \cite{Bot,JS,J} for detailed exposition of this equation.
In particular, \cite{Bot} contains the physical derivation of the
equation. Our outline follows mostly that of \cite{JS}. 

In principle, the Maxwell--Stefan equation describes the diffusive
behavior of multi-component gaseous mixtures. Consider a system consisting
of $n$ components and $u_{i}$, $1\le i\le n$, denotes the molar
concentration of the $i$th component. We assume $\sum_{i=1}^{n}u_{i}\equiv1$
so that the dimension of equation is $n-1$. Denote by $J_{i}$ the
flux of the $i$th component, and assume $\sum_{i=1}^{n}J_{i}\equiv0$.
Finally, $D_{ij}=D_{ji}>0$, $i\neq j$, represents the constant binary
diffusion coefficient between $i$th and $j$th components. Then the
Maxwell--Stefan equation is given by 
\begin{align}
\begin{cases}
\frac{\partial u_{i}}{\partial t}=-\nabla\cdot J_{i}\\
\nabla u_{i}=-\sum_{j:j\neq i}\frac{u_{j}J_{i}-u_{i}J_{j}}{D_{ij}}
\end{cases} & \;;\;i=1,\,2,\,\cdots,\,n\;.\label{e46}
\end{align}
Note that, at least in principle, we can recover $(J_{1},\cdots,J_{n})$
as a function of $(u_{1},\cdots,u_{n})$ from the second equation,
and by inserting this result into the first equation, we can derive
an equation for $(u_{1},\cdots,u_{n})$. This procedure can be explicitly
carried out for the ternary system, i.e., $n=3$. For this case, by
the elementary computation that we explained previously, it can be
shown that $\mathbf{u=}(u_{1},u_{2})$ (recall that $u_{3}=1-u_{1}-u_{2}$)
satisfies \eqref{e44} with the cross diffusion matrix 
\begin{equation}
A(u_{1},u_{2})=\frac{1}{f(u_{1},u_{2})}\begin{pmatrix}D_{23}+(D_{12}-D_{23})u_{1} & (D_{12}-D_{13})u_{1}\\
(D_{12}-D_{23})u_{2} & D_{13}+(D_{12}-D_{13})u_{2}
\end{pmatrix}\label{e48}
\end{equation}
where 
\begin{equation}
f(u_{1},u_{2})=D_{13}D_{23}+D_{13}(D_{12}-D_{23})u_{1}+D_{23}(D_{12}-D_{13})u_{2}.\label{e49}
\end{equation}
This kind of simple derivation procedure is invalid for $n\ge4$,
see \cite[Section 2]{J} for details. 

The hydrodynamic limit equation \eqref{em1} is equivalent to the
ternary Maxwell--Stefan equation described previously under the condition
that $D_{12}>D_{13},\,D_{23}$. To this end, first, for the given
ternary Maxwell--Stefan equation, let $k>0$ be an arbitrary real
number and let 
\begin{eqnarray}
 &  & \rho_{1}=kD_{13}(D_{12}-D_{23})u_{1}\;,\;\;\rho_{2}=kD_{23}(D_{12}-D_{13})u_{2}\;,\nonumber \\
 &  & \quad\sigma_{1}^{2}=D_{13}^{-1}\;,\;\;\sigma_{2}^{2}=D_{23}^{-1}\;,\;\mbox{and}\;\;\lambda=kD_{13}D_{23}\;\label{e410}
\end{eqnarray}
Then, $\rho_{1}$ and $\rho_{2}$ can be easily verified to satisfy
equation \eqref{em1}. On the other hand, for the given hydrodynamic
limit equation \eqref{em1}, let $D_{12}$ be any number larger than
$\max\{\sigma_{1}^{-2},\sigma_{2}^{-2}\}$, and let 
\begin{equation}
u_{1}=\left[\lambda(\sigma_{2}^{2}D_{12}-1)\right]^{-1}\rho_{1}\;,\;\;u_{2}=\left[\lambda(\sigma_{1}^{2}D_{12}-1)\right]^{-1}\rho_{2}\;,\;\;D_{13}=\sigma_{1}^{-2}\;,\;\mbox{and\;\;}D_{23}=\sigma_{2}^{-2}.\label{e411}
\end{equation}
Then, $u_{1}$ and $u_{2}$ can be observed to satisfy the ternary
Maxwell--Stefan equation. Owing to the multi-component nature of our
model, this equivalence is quite natural. Hence, we can reduce the
uniqueness problem of equation \eqref{em1} to that of Maxwell--Stefan
equation. 

An important feature of the Maxwell--Stefan equation is its \textit{entropy
structure} as a consequence of the Onsager reciprocity. More precisely,
the diffusion matrix \eqref{e48} can be written as $A(\cdot,\cdot)=K(\cdot,\cdot)\chi(\cdot,\cdot)$
where $\chi$ is the Hessian of entropy functional, and $K$ is a
positive-definite and symmetric matrix. The normal ellipticity of
the diffusion matrix naturally follows from this structure. Recently,
the cross-diffusion equations under the presence of the entropy structure
have been investigated by several articles. For instance, the global
existence of the weak solution and its exponential decay to the steady
state is proven in \cite{JS}, and the boundedness of this global
solution is established in \cite{J}. However, the global uniqueness
of with a general initial condition is known to be a delicate problem
(cf. \cite[Section 6]{J}), and the global uniqueness of the Maxwell--Stefan
equation for a general class of initial conditions is an open problem.
At the time when this paper is written, the global uniqueness of the
Maxwell--Stefan equation is known only for the near-equilibrium case
by \cite{Kaw}. In our context, this result can be stated in the following
manner.
\begin{thm}
\label{t4}Under Assumption \ref{a2}, there exists a unique global
weak solution of (\ref{em1}).
\end{thm}
The proof along with the non-trivial way to derive the constant $\epsilon(\lambda,\sigma_{1},\sigma_{2})$
is summarized in \cite[Section 9.4]{Bot}. This proves the part (2)
of Theorem \ref{t1}.

\subsection*{Acknowledgement }

The author would like to thank Professor S. R. S. Varadhan for numerous
valuable discussions, and Professor Ansgar J\"{u}ngel of Vienna University
of Technology for many helpful comments regarding the cross-diffusion
equation, especially for the material in Section \ref{s4}. The author
also acknowledges the support from the Samsung Scholarship.

\end{document}